\documentclass[small]{article}
\usepackage{amsmath,amstext,amsthm,amsfonts}
\usepackage{makeidx}
\usepackage{amssymb,latexsym}
\usepackage{txfonts, pxfonts}
\usepackage[english]{babel}
\usepackage[pdftex]{graphicx}
\usepackage{xcolor}
\numberwithin{equation}{section}
\newtheorem{theorem}{\textbf{Theorem}}[section]
\newtheorem{proposition}[theorem]{\textbf{Proposition}}

\newtheorem{lemma}[theorem]{\textbf{Lemma}}
\newtheorem{corollary}[theorem]{Corollary}

\theoremstyle{definition}
\newtheorem{definition}[theorem]{\textbf{Definition}}
\theoremstyle{remark}
\newtheorem{remark}[theorem]{\it{Remark}}
\newenvironment{notation}[1][Notation.]{\begin{trivlist}
\item[\hskip \labelsep {\bfseries #1}]}{\end{trivlist}}

\def\e{\epsilon}
\def\ei{\epsilon_i}
\def\R{\mathbb{R}}
\def\Rn{{\mathbb{R}}^n_+}

\def\d{\partial}

\def\fermi{\psi_i:B^+_{\delta}(0)\to M}
\def\fermilinha{\psi_i:B^+_{\delta'}(0)\to M}
\def\a{\alpha}
\def\b{\beta}
\def\l{\lambda}

\def\Bei{B^+_{\delta\ei^{-1}}}
\def\Beilinha{B^+_{\delta'\ei^{-1}}}

\def\ba{\begin{align}}
\def\ea{\end{align}}
\def\bp{\begin{proof}}
\def\ep{\end{proof}}

\def\cmedia{h}
\def\ds{d\sigma}

\begin{document}

\title{A compactness theorem for scalar-flat metrics on 3-manifolds with boundary}
\date{}

\author{\textsc{S\'ergio Almaraz}\footnote{Supported by CNPq/Brazil grant 309007/2016-0 and CAPES/Brazil grant 88881.169802/2018-01.}, \textsc{Olivaine S. de Queiroz}\footnote{Supported by CNPq/Brazil grant 310983/2017-7.}, \textsc{and Shaodong Wang}}

\maketitle

\begin{abstract}
Let $(M,g)$ be a compact Riemannian three-dimensional manifold with boundary. We prove the compactness of the set of scalar-flat metrics which are in the conformal class of $g$ and have the boundary as a constant mean curvature hypersurface. This involves a blow-up analysis of a Yamabe-type equation with critical Sobolev exponent on the boundary.
\end{abstract}

\noindent\textbf{Keywords:} compactness theorem, boundary Yamabe problem, manifolds with boundary, mean curvature.

\noindent\textbf{Mathematics Subject Classification 2000:}  53C21, 35J65, 35R01.

\section{Introduction}

Let $(M,g)$ be a Riemannian n-dimensional manifold with boundary $\d M$, and let $\nabla$ be its Riemannian connection. Denote by $R_g$ its scalar curvature and by $\Delta_g$ its Laplace-Beltrami operator, which is the Hessian trace. By $h_g$ we denote the boundary mean curvature with respect to the inward normal vector $\eta$, i.e. $h_g=\frac{1}{n-1}\sum_{i=1}^{n-1}g(\nabla_{e_i} e_i, \eta)$ for any orthonormal frame $\{e_i\}_{i=1}^{n-1}$ of $\d M$.   

In this paper we study the question of compactness of the full set of positive solutions to the equations
\begin{align}\label{main:eq}
\begin{cases}
L_{g}u=0,&\text{in}\:M,
\\
B_{g}u+Ku^{p}=0,&\text{on}\:\partial M,
\end{cases}
\end{align}
where $1< p\leq \frac{n}{n-2}$ and $K>0$ is a constant. Here, $L_g=\Delta_g-\frac{n-2}{4(n-1)}R_g$ is the conformal Laplacian and $B_g=\frac{\d}{\d\eta}-\frac{n-2}{2}h_g$ is the conformal boundary operator.

These equations have a very interesting geometrical meaning when $p=\frac{n}{n-2}$. A solution $u>0$ of \eqref{main:eq} represents a conformal metric $\tilde g=u^{\frac{4}{n-2}}g$ with scalar curvature $R_{\tilde g}=0$ and boundary mean curvature $h_{\tilde g}=\frac{2}{n-2}K$, as \eqref{main:eq} becomes a particular case of the well known equations
\begin{align*}
\begin{cases}
L_{g}u+\frac{n-2}{4(n-1)}R_{\tilde g}u^{\frac{n+2}{n-2}}=0,&\text{in}\:M,
\\
B_{g}u+\frac{n-2}{2}h_{\tilde g}u^{\frac{n}{n-2}}=0,&\text{on}\:\partial M.
\end{cases}
\end{align*}
The existence of those metrics was first studied by Escobar \cite{escobar3} motivated by the classical Yamabe problem on closed manifolds. Regularity of solutions was obtained by Cherrier in \cite{cherrier}.

The equations \eqref{main:eq} have a variational formulation in terms of the functional
\begin{align*}
Q(u)=\frac{\int_M|\nabla_gu|^2+\frac{n-2}{4(n-1)}R_gu^2dv_g+\frac{n-2}{2}\int_{\d M}h_gu^2d\sigma_g}
{\left(\int_{\d M}|u|^{p+1}d\sigma_g\right)^{\frac{2}{p+1}}}\,,
\end{align*}
where $dv_g$ and $d\sigma_g$ denote the volume forms of $M$ and $\d M$, respectively. A function $u$ is a critical point for $Q$ if and only if it solves \eqref{main:eq}. However, direct methods fail to work when $p=\frac{n}{n-2}$, as $p+1=\frac{2(n-1)}{n-2}$ is critical for the Sobolev trace embedding $H^1(M)\hookrightarrow L^{p+1}(\d M)$.
This functional has also a geometrical meaning for the critical exponent case as it becomes 
\begin{align*}
Q(u)=
\frac
{\frac{n-2}{4(n-1)}\int_MR_{\tilde g}dv_{\tilde g}+\frac{n-2}{2}\int_{\d M}h_{\tilde g}d\sigma_{\tilde g}}
{\text{area}_{\tilde g}(\d M)^{\frac{n-2}{n}}}.
\end{align*}

Defining the conformal invariant 
\begin{align*}
Q(M,\d M)=\inf\,\big\{Q(u);\:u\in C^1(\bar{M}), u\nequiv 0 \:\text{on}\: \d M\big\}\,,
\end{align*}
Escobar \cite{escobar2} observed that, when finite, $Q(M,\d M)$ has the same sign of the first eigenvalue $\l_1(B_g)$ of the problem
\begin{align*}
\begin{cases}
L_{g}u=0,&\text{in}\:M,
\\
B_{g}u+\l u=0,&\text{on}\:\partial M\,.
\end{cases}
\end{align*}
 If $\l_1(B_g)<0$, the solution of the equations (\ref{main:eq}) is unique. If $\l_1(B_g)=0$, the equations (\ref{main:eq}) become linear and the solutions are unique up to a multiplication by a positive constant. Hence, the only interesting case is the positive one.
 
When $M$ is conformally equivalent to the unit ball $B^n$, the solutions of \eqref{main:eq} are well known. The only nontrivial examples occur when $p=\frac{n}{n-2}$ and they all represent metrics isometric to the Euclidean one \cite{escobar1}. In his case, the conformal diffeomorphisms of the ball produces a blowing-up family of solutions to \eqref{main:eq}.

Working in dimension $n=3$,  our main result extends to the general case the work of Felli and Ould Ahmedou \cite{ahmedou-felli2}, that established compactness of the set of solutions to \eqref{main:eq} when $\d M$ is umbilic.

\begin{theorem}\label{compactness:thm}
Let $(M,g)$ be a Riemannian 3-manifold with boundary $\d M$.  Suppose that $Q(M,\d M)>0$ and  $M$ is not conformally equivalent to the unit ball. 
Then, given a small $\gamma_0>0$, there exists $C(M,g,\gamma_0)>0$ such that for any  $p\in\left[1+\gamma_0,\frac{n}{n-2}\right]$ and any solution $u>0$ of (\ref{main:eq})
we have
$$C^{-1}\leq u\leq C\:\:\:\:\: \text{and}\:\:\:\:\:\|u\|_{C^{2,\a}(M)}\leq C\,,$$
for some $0<\a<1$.
\end{theorem}

The subcritical Sobolev exponents $p<\frac{n}{n-2}$ in Theorem \ref{compactness:thm} provide a connection with the linear case.
Although we omit the argument (see \cite{almaraz3, ahmedou-felli1, ahmedou-felli2, han-li}), a proof of existence of a solution to Escobar's problem \cite{escobar3} can be achieved by computing the Leray-Schauder degree of all solutions of equations \eqref{main:eq}.

In the case of manifolds without boundary, the question of compactness of the full set of smooth solutions to the Yamabe equation was first raised by R. Schoen in a topics course at Stanford University in 1988. A necessary condition is that the manifold $M^n$ is not conformally equivalent to the sphere $S^n$. 
This problem was studied in  \cite{druet1, druet2, li-zhang, li-zhang2, li-zhu2, marques, schoen4, schoen-zhang} and was completely solved in \cite{brendle2, brendle-marques, khuri-marques-schoen}. In \cite{brendle2}, Brendle discovered the first smooth counterexamples for dimensions $n\geq 52$ (nonsmooth examples were obtained by Ambrosetti and Malchiodi in \cite{berti-malchiodi}). In \cite{khuri-marques-schoen}, Khuri, Marques and Schoen proved compactness for dimensions $n\leq 24$. Their proof contains both a local and a global aspect. The local aspect involves the vanishing of the Weyl tensor up to order $[\frac{n-6}{2}]$ at any blow-up point and the global aspect involves the positive mass theorem. Finally, in \cite{brendle-marques}, Brendle and Marques extended the counterexamples of \cite{brendle2} to the remaining dimensions  $25\leq n\leq 51$. In the case of nonempty umbilical boundary, the same compactness and noncompactness results were obtained by Disconzi and Khuri in \cite{disconzi-khuri} for the boundary condition $B_gu=0$.

Despite its additional technical difficulties, the question of compactness of the solutions of \eqref{main:eq} turns out to have great similarity with the one above for the classical Yamabe equation. In \cite{ahmedou-felli1} Felli and Ould Ahmedou prove compactness for locally conformally flat manifolds with umbilic boundary, a result previously obtained by Schoen  \cite{schoen4} for the classical Yamabe equation. In \cite{almaraz3} the first author proves the vanishing of the trace-free boundary second fundamental form in dimensions $n\geq 7$ at any blow-up point, a result inspired by the vanishing of the Weyl tensor in dimensions $n\geq 6$ obtained by Li-Zhang and Marques independently in  \cite{li-zhang, li-zhang2, marques}. On the other hand, the noncompactness results of Brendle and Marques inspired the first author's paper \cite{almaraz5} which provides counterexamples in dimensions $n\geq 25$ to compactness in \eqref{main:eq}.  So Theorem \ref{compactness:thm} ensures that there is a critical dimension $3<n_0\leq 25$ such that compactness for the set of positive smooth solutions of \eqref{main:eq} holds for $n< n_0$ and fails for $n\geq n_0$.

Although the corresponding result for the classical Yamabe equation in dimension $3$ was obtained  by Li and Zhu in \cite{li-zhu2}, our approach to Theorem \ref{compactness:thm} makes use of some further techniques of the later works \cite{ khuri-marques-schoen, marques}. This is because the canonical bubble, coming from the Euclidean metric on $B^3$, fails to provide a good approximation for the blowing up solutions of \eqref{main:eq}.

The strategy of the proof of Theorem \ref{compactness:thm} is  similar to the one proposed by Schoen in the case of manifolds without boundary.  It is based on finding local obstructions to blow-up by means of a Pohozaev-type identity. 
Assuming that a sequence $\{u_i\}$ of solutions has an isolated simple blow-up point, we approximate $\{u_i\}$ by the standard Euclidean solution plus a correction term $\phi_i$. The function $\phi_i$ is defined as a solution to a non-homogeneous linear equation and is similar to the one in \cite{khuri-marques-schoen}. We then use the Pohozaev identity to prove a local sign restriction in dimension three, which allows the reduction to the simple blow-up case. This sign restriction is used again to derive a contradiction with the positive mass theorem established in \cite{almaraz-barbosa-lima} for manifolds modeled on the Euclidean half-space.

A key point in dimension three is that this hypothesis simplifies the estimates on the right side of the Pohozaev identity as every geometric term, including $\phi_i$, only contributes to the high order terms in the proof of the local sign restriction.  It contrasts with the case of higher dimensions where further estimates on the geometric terms would be needed. Another point that differs from the mentioned papers on compactness is that we only use a very rough control of the Green's function. The relation with the positive mass theorem comes from an integral expression obtained by Brendle-Chen in \cite{brendle-chen}.

This paper is organized as follows. In Section \ref{sec:pre} we present some preliminary computations about the standard solution on the Euclidean half-space, Fermi coordinates and the conformal invariant equation associated to \eqref{main:eq}. 
The important Pohozaev identity and the mass term is studied in Section \ref{sec:poho:mass}. The definition of isolated and isolated simple blow-up points and some additional properties are collected in Section \ref{sec:isolated:simple}, while the blow-up estimates are presented in Section \ref{sec:blowup:estim}. In Section \ref{sec:sign:restr} we come back to the Pohozaev integral and prove the sign restriction and some of its consequences. Finally we give a proof of the main result in Section \ref{sec:pf:thm}.

\section{Preliminaries}\label{sec:pre}

\subsection{Notations}\label{subsec:notations}
Throughout this work we will make use of the index notation for tensors, commas denoting covariant differentiation. We will adopt the summation convention whenever confusion is not possible. When dealing with coordinates on manifolds with boundary, we will use indices $1\leq i,j,k,l\leq n-1$ and $1\leq a,b,c,d\leq n$. In this context, lines under or over an object mean the restriction of the metric to the boundary is involved.

We will denote by $g$ the Riemannian metric and set $\det g=\det g_{ab}$. The induced metric on $\d M$ will be denoted by $\bar{g}$. We will denote by $\nabla_g$  the covariant derivative and by $\Delta_g$ the Laplacian-Beltrami operator. By $R_g$ or $R$ we will denote the scalar curvature. The second fundamental form of the boundary will be denoted by $\pi_{kl}$ and the mean curvature,  $\frac{1}{n-1}tr (\pi_{kl})$, by $h_g$ or $h$.  

By $\Rn$ we will denote the half-space $\{z=(z_1,...,z_n)\in \R^n;\:z_n\geq 0\}$. If $z\in\Rn$ we set $\bar{z}=(z_1,...,z_{n-1})\in\R^{n-1}\cong \d\Rn$. 
We define $B^+_{\delta}(0)=\{z\in\Rn\,;\:|z|<\delta \}$. We also denote $B^+_{\delta}=B^+_{\delta}(0)$ for short. 
We set $\d^+B^+_{\delta}(0)=\d B^+_{\delta}(0)\cap \Rn=\{z\in\Rn\,;\:|z|=\delta \}$ and $\d 'B^+_{\delta}(0)=B^+_{\delta}(0)\cap \d\Rn=\{z\in\d\Rn\,;\:|z|<\delta \}$. Thus, $\d B^+_{\delta}(0)=\d 'B^+_{\delta}(0)\cup \d^+B^+_{\delta}(0)$.

In various parts of the text, we will make use of Fermi coordinates (see Definition \ref{def:fermi} below) $$\psi:B^+_{\delta}(0)\to M$$ centered at a point $x_0\in\d M$. In this case, we will work in $B^+_{\delta}(0)\subset \Rn$. 

\subsection{Standard solutions in the Euclidean half-space}\label{subsec:standard_solutions}

In this subsection we study the Euclidean Yamabe equation in $\Rn$ and its linearization.

The simplest example of solution to the Yamabe-type problem we are concerned is the ball in $\R^n$ with the canonical Euclidean metric. This ball is conformally equivalent to the half-space $\Rn$ by the inversion 
$F:\mathbb{R}_+^n\to B^n\backslash\{ (0,...,0,-1)\}$
with respect to the sphere with center $(0,...,0,-1)$ and radius $1$. Here, $B^n$ is the Euclidean ball in $\R^{n}$ with center $(0,...,0,-1/2)$ and radius $1/2$. The expression for $F$ is
$$F(y_1,...y_n)=\frac{(y_1,...,y_{n-1},y_n+1)}{y_1^2+...+y_{n-1}^2+(y_n+1)^2}+(0,...,0,-1)\,,$$
and of course  its inverse mapping $F^{-1}$ has the same expression.
An easy calculation shows that $F$ is a conformal map and $F^*g_{eucl} = U^{\frac{4}{n-2}}g_{eucl}$ in $\Rn$, where $g_{eucl}$ is the Euclidean metric and  
$U(y)=(y_1^2+...+y_{n-1}^2+(y_n+1)^2)^{-\frac{n-2}{2}}$.
The function $U$ satisfies
\begin{equation}
\label{eq:U}
\begin{cases}
\Delta U = 0\,,&\text{in}\:\mathbb{R}_+^n\,,\\
\frac{\partial U}{\partial y_n}+(n-2)U^{\frac{n}{n-2}}=0\,,&\text{on}\:\partial\mathbb{R}_+^n\,.
\end{cases}
\end{equation}

Since the equations (\ref{eq:U}) are invariant by horizontal translations and scalings with respect to the origin, we obtain the following family of solutions of (\ref{eq:U}):
\begin{equation}\label{fam:U}
U_{\l,z}(y)=\left(\frac{\l}{(\l+y_n)^2+\sum_{j=1}^{n-1}(y_j-z_j)^2}\right)^{\frac{n-2}{2}}\,,
\end{equation}
where $\l>0$ and $z=(z_1,...,z_{n-1})\in \R^{n-1}$.

In fact, the converse statement is also true: by a Liouville-type theorem in \cite{li-zhu} (see also \cite{chipot-shafrir-fila, escobar1}), any non-negative solution to the equations (\ref{eq:U}) is of the form (\ref{fam:U}) or is identically zero.

The existence of the family of solutions (\ref{fam:U}) has two important consequences. First, we see that the set of solutions of the equations (\ref{eq:U}) is non-compact. In particular, the set of solutions of (\ref{main:eq}) with $p=\frac{n}{n-2}$ is not compact when $M^n$ is conformally equivalent to $B^n$. Secondly, the functions $\frac{\d U}{\d y_j}$, for $j=1,...,n-1$, and $\frac{n-2}{2}U+y^b\frac{\d U}{\d y^b}$, are solutions to the following homogeneous linear problem:
\ba
\begin{cases}\label{linear:homog}
\Delta\psi=0\,,&\text{in}\:\Rn\,,
\\
\frac{\d\psi}{\d y_n}+nU^{\frac{2}{n-2}}\psi=0\,,&\text{on}\:\d\Rn\,.
\end{cases}
\end{align}
\begin{lemma}
\label{classifLinear}
Suppose $\psi$ is a solution to
\begin{equation}
\label{eq:Ulinear}
\begin{cases}
\Delta \psi = 0\,,&\text{in}\:\mathbb{R}_+^n\,\\
\frac{\partial \psi}{\partial y_n}+nU^{\frac{2}{n-2}}\psi=0\,,&\text{on}\:\partial \mathbb{R}_+^n\,.
\end{cases}
\end{equation}
If $\psi(y)=O((1+|y|)^{-\alpha})$ for some $\alpha>0$, then there exist constants $c_1,...,c_{n}$ such that
\[
\psi (y)=\sum_{i=1}^{n-1} c_i\frac{\d U}{\d y_j}+c_n\Big(\frac{n-2}{2}U+y^b\frac{\d U}{\d y^b}\Big)\,.
\]
\end{lemma}
\bp
This is \cite[Lemma 2.1]{almaraz3}.
\ep


\subsection{Coordinate expansions for the metric}\label{subsec:coordinate}

Recall the definition of Fermi coordinates:
\begin{definition}\label{def:fermi}
Let $x_0\in\d M$ and choose boundary geodesic normal coordinates $(z_1,...,z_{n-1})$, centered at $x_0$, of the point $x\in\d M$.
We say that $z=(z_1,...,z_n)$, for small $z_n\geq 0$, are the {\it{Fermi coordinates}} (centered at $x_0$) of the point $\exp_{x}(z_n\eta(x))\in M$. Here, we denote by $\eta(x)$ the inward unit normal vector to $\d M$ at $x$. In this case, we have a map  $\psi(z)=\exp_{x}(z_n\eta(x))$, defined on a subset of $\Rn$.
\end{definition}
It is easy to see that in these coordinates $g_{nn}\equiv 1$ and $g_{jn}\equiv 0$, for $j=1,...,n-1$.
The expansion for $g$ in Fermi coordinates is given by:
	\begin{align}\label{exp:g}
	g_{ij}(\psi(z))=\delta_{ij}-2\pi_{ij}(x_0)z_n+O(|z|^2),\notag
	\\
	g^{ij}(\psi(z))=\delta_{ij}+2\pi_{ij}(x_0)z_n+O(|z|^2).
	\end{align}

The existence of conformal Fermi coordinates, introduced in \cite{coda1}, is stated as follows:
\begin{proposition}\label{conf:fermi:thm}
For any given integer $N\geq 1$, there is a metric $\tilde{g}$, conformal to $g$, such that in $\tilde{g}$-Fermi coordinates $\tilde{\psi}:B^+_{\delta}(0)\to M$ centered at $x_0$, we have
$$
(\det \tilde{g})(\tilde{\psi}(z))=1+O(|z|^N)\,.
$$  
Moreover, $\tilde{g}$ can be written as $\tilde{g}=fg$, where $f$ is a positive function with $f(x_0)=1$ and $\displaystyle\frac{\d f}{\d z_k}(x_0)=0$ for $k=1,...,n-1$. In this metric we also have $h(\tilde{\psi}(z))=O(|z|^{N-1})$.
\end{proposition}

\bp
The first part is \cite[Proposition 3.1]{coda1} and the last one follows from
\begin{equation*}
h_g=\frac{-1}{2(n-1)}g^{ij}g_{ij,\,n}=\frac{-1}{2(n-1)}(\log\det g)_{,\,n}\,.
\end{equation*}
\ep


\subsection{Conformal scalar and mean curvature equations}\label{subsec:conformal_scalar}
In this subsection we study the partial differential equation we will work with in the next sections:
\begin{align}\label{eq:u:geral}
\begin{cases}
L_{g}u=0,&\text{in}\:M\,,
\\
B_{g}u+(n-2)f^{-\tau}u^{p}=0,&\text{on}\:\d M\,,
\end{cases}
\end{align}
where $\tau=\frac{n}{n-2}-p$, $1+\gamma_0\leq p\leq\frac{n}{n-2}$ for some fixed $\gamma_0>0$ and $f$ is a positive function.

The equations (\ref{eq:u:geral}) have an important scaling invariance property. Fix $x_0\in \d M$ and take $\delta>0$ small, and consider Fermi coordinates $\psi:B^+_{\delta}(0)\to M$ centered at $x_0$. Given $s>0$ we define the renormalized function 
$$
v(y)=s^{\frac{1}{p-1}}u(\psi(sy))\,,\:\:\:\text{for}\: y\in B^+_{\delta s^{-1}}(0)\,.
$$ 
Then
\begin{align}\notag
\begin{cases}
L_{\hat{g}}v=0,&\text{in}\:B^+_{\delta s^{-1}}(0)\,,
\\
B_{\hat{g}}v+(n-2)\hat{f}^{-\tau}v^{p}=0,&\text{on}\: \d 'B^+_{\delta s^{-1}}(0)\,,
\end{cases}
\end{align}
where $\hat{f}(y)=f(\psi(sy))$ and the metric $\hat{g}$ is defined by $\hat{g}_{kl}(y)=g_{kl}(\psi(sy))$.

The reason to work with the equations (\ref{eq:u:geral}), instead of (\ref{main:eq}), is that they have important conformal invariance properties.  Suppose $\tilde{g}=\zeta^{\frac{4}{n-2}}g$ is a metric conformal to $g$.
It follows from the properties
\begin{equation*}
L_{\zeta^{\frac{4}{n-2}}g}(\zeta^{-1}u)=\zeta^{-\frac{n+2}{n-2}}L_{g}u
\:\:\:\:\text{and}\:\:\:\:
B_{\zeta^{\frac{4}{n-2}}g}(\zeta^{-1}u)=\zeta^{-\frac{n}{n-2}}B_{g}u
\end{equation*}
that, if $u$ is a solution of the equations (\ref{eq:u:geral}), then $\zeta^{-1}u$ satisfies
\begin{align}
\begin{cases}\notag
L_{\tilde{g}}(\zeta^{-1}u)=0,&\text{in}\:M\,,
\\
B_{\tilde{g}}(\zeta^{-1}u)+(n-2)(\zeta f)^{-\tau}(\zeta^{-1}u)^{p}=0,&\text{on}\:\d M\,,
\end{cases}
\end{align}
which are again equations of the same type. 
\begin{notation}
Let $\Omega\subset M$ be a domain in a Riemannian manifold $(M,g)$.
Let $\{g_i\}$ be a sequence of metrics on $\Omega$. We say that $u_i\in\mathcal{M}_i$ if $u_i>0$ satisfies
\begin{align}\label{eq:ui}
\begin{cases}
L_{g_i}u_i=0,&\text{in}\:\Omega\,,
\\
B_{g_i}u_i+(n-2)f_i^{-\tau_i}u_i^{p_i}=0,&\text{on}\:\Omega\cap \d M\,,
\end{cases}
\end{align}
where $\tau_i=\frac{n}{n-2}-p_i$ and $1+\gamma_0\leq p_i\leq\frac{n}{n-2}$ for some fixed $\gamma_0>0$. 
\end{notation}

In many parts of this article we will work with sequences $\{u_i\in\mathcal{M}_i\}_{i=1}^{\infty}$. In this case, we assume that $f_i\to f$ in the $C^1_{loc}$ topology, for some positive function $f$, and that $g_i\to g_0$ in the $C^3_{loc}$ topology, for some metric $g_0$. 

By the conformal invariance stated above, we are allowed to replace the metric $g_i$ by $\zeta_i^{\frac{4}{n-2}}g_i$ as long as we have control of the conformal factors $\zeta_i$. In this case, we replace the sequence $\{u_i\}$ by $\{\zeta_i^{-1}u_i\}$. In particular, we can use conformal Fermi coordinates (see Proposition \ref{conf:fermi:thm}) centered at some point $x_i\in\d M$, as those conformal changes are uniformly controlled with respect to $i$ by construction.


\section{The Pohozaev identity and the mass term}\label{sec:poho:mass}
Let $g$ be a Riemannian metric on the half-ball $B_{\delta}^+(0)$. For any $x=(x_1,...,x_n)\in \R^n$ we set $r=|x|=\sqrt{x_1^2+...+x_n^2}$. For any smooth function $u$ on $B^+_{\delta}(0)$ and $0<\rho<\delta$ we define
$$P(u,\rho)=\int_{\partial^+ B^+_{\rho}(0)}\left(\frac{n-2}{2}u\frac{\partial u}{\partial r}-\frac{r}{2}|du|^2+r\left|\frac{\partial u}{\partial r}\right|^2\right)d\sigma
+\frac{\rho}{p+1}\int_{\partial\,(\partial' B^+_{\rho}(0))}Kf^{-\tau}u^{p+1}d\bar{\sigma}$$
and
$$P'(u,\rho)=\int_{\partial^+ B^+_{\rho}(0)}\left(\frac{n-2}{2}u\frac{\partial u}{\partial r}-\frac{r}{2}|du|^2+r\left|\frac{\partial u}{\partial r}\right|^2\right)d\sigma\,.$$

An integration by parts \cite[Proposition 3.1]{almaraz3} gives the following Pohozaev-type identity to be used in the analysis of blow-up sequences:
\begin{proposition}\label{Pohozaev}
If $u$ is a solution of
\begin{equation}\notag
\begin{cases}
\Delta_{g}u-\frac{n-2}{4(n-1)}R_gu=0\,,&\text{in}\:B^+_{\delta}(0)\,,\\
\partial_n u-\frac{n-2}{2} h_gu+Kf^{-\tau}u^{p}=0\,,&\text{on}\:\partial' B^+_{\delta}(0) \,,
\end{cases}
\end{equation}
where $K$ is a constant, then
\begin{align}
P(u,\rho)&=-\int_{B_{\rho}^+(0)}\left(x^a\partial_a u+\frac{n-2}{2}u\right)A_g(u)dx
+\frac{n-2}{2}\int_{\partial' B^+_{\rho}(0)}\left(\bar{x}^k\partial_k u+\frac{n-2}{2}u\right)h_gud\bar{x}\notag
\\
&-\frac{\tau}{p+1}\int_{\partial' B^+_{\rho}(0)}K(\bar{x}^k\partial_k f)f^{-\tau-1}u^{p+1}d\bar{x}
+\left(\frac{n-1}{p+1}-\frac{n-2}{2}\right)\int_{\partial' 
B^+_{\rho}(0)}Kf^{-\tau}u^{p+1}d\bar{x}\,,\notag
\end{align}
where $A_g=\Delta_g-\Delta-\frac{n-2}{4(n-1)}R_g$. Here, $\Delta$ stands for the Euclidean Laplacian.
\end{proposition}

While in Section \ref{sec:sign:restr} we will obtain a sign restriction for $P'(u,\rho)$ by means of Proposition \ref{Pohozaev}, in this section we handle $P'(u,\rho)$ directly and relate it with a mass-type geometric invariant defined below.
\begin{lemma}\label{propo:P'}
If $\phi(x)=u(x)-|x|^{2-n}$ then
\begin{align*}
P'(u,\rho)&=\frac{n-2}{2}\int_{\partial^+ B^+_{\rho}(0)}\left(\frac{\partial}{\partial r}r^{2-n}\phi(x)-r^{2-n}\frac{\partial \phi}{\partial r}\right)d\sigma
\\
&+\frac{1}{2}\int_{\partial^+ B^+_{\rho}(0)}\left(r\,\Big(\frac{\d\phi}{\d r}\Big)^2-r|d\phi|^2+\frac{\d\phi}{\d r}\Big((n-2)\phi+r\frac{\d\phi}{\d r}\Big)\right) d\sigma\,.
\end{align*}
\end{lemma}
\bp
Direct calculations give
\begin{align*}
\frac{n-2}{2}u\d_r u&-\frac{r}{2}|du|^2+r(\d_ru)^2
=\frac{1}{2}(\d_ru)\Big((n-2)u+r\d_ru\Big)+\frac{r}{2}\Big((\d_ru)^2-|du|^2\Big)
\\
&=\frac{1}{2}\big((n-2)\d_rr^{2-n}\phi-(n-2)r^{2-n}\d_r\phi+(n-2)\phi\d_r\phi+r(\d_r\phi)^2\big)
\\
&\hspace{1cm}+\frac{r}{2}\big((\d_r\phi)^2-|d\phi|^2\big),
\end{align*}
from which the result follows.
\ep
\begin{definition}\label{def:asym}
Let $(\hat M, g)$ be a Riemannian manifold with a noncompact  boundary $\d \hat M$. 
We say that $\hat M$ is {\it{asymptotically flat}} with order $q>0$, if there is a compact set $K\subset \hat M$ and a diffeomorphism $f:\hat M\backslash K\to \Rn\backslash \overline{B^+_1}$ such that, in the coordinate chart defined by $f$ (which we call the {\it  asymptotic coordinates} of $\hat M$), we have
$$
|g_{ab}(y)-\delta_{ab}|+|y||g_{ab,c}(y)|+|y|^2|g_{ab,cd}(y)|=O(|y|^{-q})\,,
\:\:\:\:\text{as}\:\:|y|\to\infty\,,
$$
where $a,b,c,d=1,...,n$.
\end{definition}

Suppose the manifold $\hat M$, of dimension $n\geq 3$,  is asymptotically flat with order $q>\frac{n-2}{2}$, as defined  above. Assume also that $R_g$ is integrable on $\hat M$, and $\cmedia_g$  is integrable on $\d \hat M$. Let $(y_1,...,y_n)$ be the  asymptotic coordinates induced by the diffeomorphism $f$. 
Then the limit
\begin{align}\label{def:mass}
m(g)=
\lim_{R\to\infty}\left\{
\sum_{a,b=1}^{n}\int_{y\in\Rn,\, |y|=R}(g_{ab,b}-g_{bb,a})\frac{y_a}{|y|}\,\ds
+\sum_{i=1}^{n-1}\int_{y\in\d\Rn,\, |y|=R}g_{ni}\frac{y_i}{|y|}\,\ds\right\}
\end{align}
exists, and we call it the {\it mass} of $(\hat M, g)$. As proved in \cite{almaraz-barbosa-lima}, $m(g)$ is a geometric invariant in the sense that it does not depend on the asymptotic coordinates.

The expression in \eqref{def:mass} is due to F. Marques and is the analogue of the ADM mass for the manifolds of Definition \ref{def:asym}. A positive mass theorem for $m(g)$, similar to the classical ones in \cite{schoen-yau, witten}, is stated as follows:
\begin{theorem}[\cite{almaraz-barbosa-lima}]\label{pmt}{\footnote{
Actually, the dimensional restriction  is unnecessary as Theorem \ref{pmt} can be reduced to the one for manifolds without boundary (see \cite{almaraz-barbosa-lima}).}}
Assume $n=3$.
If $R_g$, $\cmedia_g\geq 0$, then we have $m(g)\geq 0$ and the equality holds if and only if $\hat M$ is isometric to $\R_+^3$.
\end{theorem}

The asymptotically flat manifolds we work with in this paper come from the stereographic projection of compact manifolds with boundary. Inspired by Schoen's approach \cite{schoen1} to the classical Yamabe problem, this  projection is defined by means of a Green's function with singularity at a boundary point. Since we do not have the control of the Green's function expression used in the case of manifolds without boundary, the relation with \eqref{def:mass} is obtained by means of an integral defined in \cite{brendle-chen}. This is stated in the next proposition. 
\begin{proposition}\label{propo:I:mass}
Let $(M,g)$ be a compact n-manifold with boundary and consider Fermi coordinates centered at $x_0\in \d M$. If $d=\left[\frac{n-2}{2}\right]$, 
suppose in those coordinates we have 
$$
g_{ab}(x)=\delta_{ab}+h_{ab}(x)+O(|x|^{2d+2})
$$ 
with $h_{ab}(x)=O(|x|^{d+1})$ and $\text{tr}\,(h_{ab}(x))=O(|x|^{2d+2})$. Let $G$ be a smooth positive function on $M\backslash\{x_0\}$ written near $x_0$ as
$$
G(x)=|x|^{2-n}+\phi(x)
$$ 
where $\phi$ is smooth on $M\backslash\{x_0\}$ satisfying $\phi(x)=O(|x|^{d+3-n}|\log|x||)$. If we define the metric $\hat g=G^{\frac{4}{n-2}}g$ and set 
\begin{align*}
I(x_0,\rho)
&=\frac{4(n-1)}{n-2}\int_{\d^+B^+_{\rho}(0)}\left(|x|^{2-n}\d_aG(x)-\d_a|x|^{2-n}G(x)\right)\frac{x_a}{|x|}d\sigma
\\
&-\int_{\d^+B^+_{\rho}(0)}\left(|x|^{3-2n}x_ a\d_bh_{ab}(x)-2n|x|^{1-2n}x_ax_bh_{ab}(x)\right)d\sigma\,,
\end{align*}
then $(M\backslash \{x_0\},\hat g)$ is asymptotically flat in the sense of Definition \ref{def:asym} with mass
$$
m(\hat g)=\lim_{\rho\to 0}I(x_0,\rho).
$$
\end{proposition}
\bp
Consider inverted coodinates $y_a=|x|^{-2}x_a$.
The first statement follows from the fact that $\hat{g}\left(\frac{\d}{\d y_a},\frac{\d}{\d y_b}\right)=\delta_{ab}+O(|y|^{-d-1}|\log |y||)$.
In order to prove the last one, we can mimic the proof of \cite[Proposition 4.3]{brendle-chen} to obtain
\ba
\int_{\d^+B^+_{\rho^{-1}}(0)}\frac{y_a}{|y|}\frac{\d}{\d y_b}
\hat{g}\left(\frac{\d}{\d y_a},\frac{\d}{\d y_b}\right)\ds_{\rho^{-1}}
&-\int_{\d^+B^+_{\rho^{-1}}(0)}\frac{y_a}{|y|}\frac{\d}{\d y_a}
\hat{g}\left(\frac{\d}{\d y_b},\frac{\d}{\d y_b}\right)\ds_{\rho^{-1}}\notag
\\
&=\mathcal{I}(x_0,\rho)+O(\rho^{2d+4-n}(\log\rho)^2)\,.\notag
\end{align}
Since $(x_1,...,x_n)$ are Femi coordinates,
$$
\hat{g}\left(\frac{\d}{\d y_i}, \frac{\d}{\d y_n}\right)=0\,,
\:\:\:\:\:\text{for}\: i=1,...,n-1\,,
\:\:\:\text{if}\:y_n=0\,
$$
the result then follows
\ep

\begin{proposition}\label{I:P'}
Under the hypotheses of Proposition \ref{propo:I:mass} we have
\begin{align*}
P'(G,\rho)=-\frac{(n-2)^2}{8(n-1)}I(x_0,\rho)+O(\rho^{2d+4-n}|\log \rho|).
\end{align*}
\end{proposition}
\bp
Since $h_{na}(x)=0$ and $\text{tr}\,(h_{ij}(x))=O(|x|^{2d+2})$, we have
$$
x_ i\d_j h_{ij}(x)=\d_j(x_ ih_{ij}(x))+O(|x|^{2d+2}).
$$
Then
\begin{align*}
\int_{\d^+B^+_{\rho}(0)}&\left(|x|^{3-2n}x_a\d_bh_{ab}(x)-2n|x|^{1-2n}x_ax_bh_{ab}(x)\right)d\sigma
\\
&=\rho^{3-2n}\int_{\d^+B^+_{\rho}(0)}\d_j(x_ih_{ij}(x))\,d\sigma+O(\rho^{2d+4-n}),
\end{align*}
and the first integral on the right hand side vanishes. A direct calculation shows
\begin{align*}
\int_{\d^+B^+_{\rho}(0)}&\left(|x|^{2-n}\d_aG(x)-\d_a|x|^{2-n}G(x)\right)\frac{x_a}{|x|}d\sigma
\\
&=\int_{\d^+B^+_{\rho}(0)}\left(|x|^{2-n}\d_a\phi(x)-\d_a|x|^{2-n}\phi(x)\right)\frac{x_a}{|x|}d\sigma\,,
\end{align*}
and so
\begin{align*}
I(x_0,\rho)=\frac{4(n-1)}{n-2}\int_{\d^+B^+_{\rho}(0)}\left(|x|^{2-n}\d_a\phi(x)-\d_a|x|^{2-n}\phi(x)\right)\frac{x_a}{|x|}d\sigma+O(\rho)\,.
\end{align*}
On the other hand, using Lemma \ref{propo:P'} we obtain
\begin{align*}
P'(G,\rho)=-\frac{n-2}{2}\int_{\d^+B^+_{\rho}(0)}\left(|x|^{2-n}\d_a\phi(x)-\d_a|x|^{2-n}\phi(x)\right)\frac{x_a}{|x|}d\sigma+O(\rho^{2d+4-n}|\log \rho|),
\end{align*}
and the result follows.
\ep

\section{Isolated and isolated simple blow-up points}\label{sec:isolated:simple}

In this section we briefly collect the definitions and main results of isolated and isolated simple blow-up sequences from \cite[Section 4]{almaraz3}. They are inspired by the corresponding ones for manifolds without boundary and are similar to the ones in \cite{ahmedou-felli1, ahmedou-felli2}.  
\begin{definition}\label{def:blow-up}
	Let $\Omega\subset M$ be a domain in a Riemannian manifold $(M,g)$.
	We say that $x_0\in\Omega\cap\d M$ is a {\it{blow-up point}} for the sequence $\{u_i\in\mathcal{M}_i\}_{i=1}^{\infty}$, if there is a sequence $\{x_i\}\subset\Omega\cap\d M$ such that 
	
	(1) $x_i\to x_0$;
	
	(2) $u_i(x_i)\to\infty$; 
	
	(3) $x_i$ is a local maximum of $u_i|_{\d M}$.\\
	Briefly we say that $x_i\to x_0$ is a blow-up point for $\{u_i\}$. The sequence $\{u_i\}$ is called a {\it{blow-up sequence}}.
\end{definition}
\noindent
{\bf{Convention}} If $x_i\to x_0$ is a blow-up point, we use $g_i$-Fermi coordinates $$\fermi$$ centered at $x_i$ and work in $B^+_{\delta}(0)\subset\Rn$, for some small $\delta>0$.
\begin{notation}
	If $x_i\to x_0$ is a blow-up point we set $M_i=u_i(x_i)$, $\ei=M_i^{-(p_i-1)}$.
\end{notation}
\begin{definition}\label{def:isolado}
	We say that a blow-up point $x_i\to x_0$ is an {\it{isolated}} blow-up point for $\{u_i\}$ if there exist $\delta,C>0$ such that 
	\begin{equation}\notag
	u_i(x)\leq Cd_{\bar{g}_i}(x,x_i)^{-\frac{1}{p_i-1}}\,,\:\:\:\:\text{for all}\: x\in\d M\backslash \{x_i\}\,,\:d_{\bar{g}_i}(x,x_i)< \delta\,.
	\end{equation}
\end{definition}
	Since Fermi coordinates are normal on the boundary, the above definition is equivalent to
	\begin{equation}\label{des:isolado}
	u_i(\psi_i(z))\leq C|z|^{-\frac{1}{p_i-1}}\,,\:\:\:\:\text{for all}\: z\in \d 'B^+_{\delta}(0)\backslash\{0\}\,.
	\end{equation}
	This definition is invariant under renormalization. This follows from the fact that if $v_i(y)=s^{\frac{1}{p_i-1}}u_i(\psi_i(sy))$, then
	$$
	u_i(\psi_i(z))\leq C|z|^{-\frac{1}{p_i-1}}\Longleftrightarrow v_i(y)\leq C|y|^{-\frac{1}{p_i-1}}\,,
	$$
	where $z=sy$.
	
	Harnack inequalities give the following two lemmas:
\begin{lemma}\label{Harnack:bordo}
	Let $x_i\to x_0$ be an isolated blow-up point. Then $\{u_i\}$ satisfies
	\begin{equation}\notag
	u_i(\psi_i(z))\leq C|z|^{-\frac{1}{p_i-1}}\,,\:\:\:\text{for all}\: z\in B^+_{\delta}(0)\backslash\{0\}\,.
	\end{equation}
\end{lemma}
\begin{lemma}\label{Harnack}
	Let $x_i\to x_0$ be an isolated blow-up point and $\delta$ as in Definition \ref{def:isolado}. 
	Then there exists $C>0$ such that for any $0<s<\frac{\delta}{3}$ we have
	\begin{equation}\notag
	\max_{B_{2s}^+(0)\backslash B_{s/2}^+(0)} (u_i\circ\psi_i)\leq C\min_{B_{2s}^+(0)\backslash B_{s/2}^+(0)} (u_i\circ\psi_i)\,.
	\end{equation}
\end{lemma}

The next proposition says that, in the case of an isolated blow-up point, the sequence $\{u_i\}$, when renormalized, converges to the standard Euclidean solution $U$.
\begin{proposition}\label{form:bolha}
	Let $x_i\to x_0$ be an isolated blow-up point. We set
	$$
	v_i(y)=M_i^{-1}(u_i\circ\psi_i)(M_i^{-(p_i-1)} y)\,,\:\:\:\: \text{for}\: y\in B^+_{\delta M_i^{p_i-1}}(0)\,.
	$$
	Then given $R_i\to\infty$ and $\b_i\to 0$, after choosing subsequences, we have  
	\\\\
	(a) $|v_{k(i)}-U|_{C^2(B^+_{R_i}(0))}<\b_i$;
	\\
	(b) $\lim_{i\to\infty}\frac{R_i}{\log M_{k(i)}}= 0$;
	\\
	(c) $\lim_{i\to\infty}p_{k(i)}=\frac{n}{n-2}$.
\end{proposition}
\begin{remark}\label{rk:mud_conf}
Let $x_i\to x_0$ and consider a conformal change $\zeta_i^{\frac{4}{n-2}}g_i$ of the metrics $g_i$ (see the last paragraph of Section \ref{subsec:conformal_scalar}). Suppose the conformal factors $\zeta_i>0$ are uniformly bounded (above and below) with $\zeta_i(x_i)=1$ and $\frac{\d \zeta_i}{\d z_k}(x_i)=0$ for $k=1,...,n-1$.
Then, using Proposition \ref{form:bolha}, it is not difficult to see that $x_i\to x_0$ is an isolated blow-up point for $\{u_i\}$ if and only it is for $\{\zeta_i^{-1}u_i\}$. This is the case when we use conformal Fermi coordinates (see Proposition \ref{conf:fermi:thm}) centered at $x_i$.
\end{remark}

The set of blow-up points is handled in the next proposition.
\begin{proposition}\label{conj:isolados}
	Given small $\b>0$ and large $R>0$ there exist constants $C_0, C_1>0$, depending only on $\b$, $R$ and $(M^n,g)$, such that if $u$ is solution of \eqref{eq:u:geral} and $\max_{\d M} u\geq C_0$, then $\frac{n}{n-2}-p<\b$ and
	there exist $x_1,...,x_N\in\d M$, $N=N(u)\geq 1$, local maxima of $u$, such that:
	\\\\
	(1) If $r_j=Ru(x_j)^{-(p-1)}$ for $j=1,...,N$,  then $\{D_{r_j}(x_j)\subset\d M\}_{j=1}^{N}$ is a disjoint collection, where $D_{r_j}(x_j)$ is the boundary metric ball.
	\\\\
	(2) For each $j=1,...,N$, 
	$\:\:\:\:
	\left|u(x_j)^{-1}u(\bar{\psi}_j(z))-U(u(x_j)^{p-1}z)\right|_{C^2(B^+_{2r_j}(0))}<\b
	$,
	\\
	where we are using Fermi coordinates $\bar{\psi}_j:B^+_{2r_j}(0)\to M$ centered at $x_j$.
	\\\\
	(3) We have
	$$
	u(x)\,d_{\bar{g}}(x,\{x_1,...,x_N\})^{\frac{1}{p-1}}\leq C_1\,,\:\:\:\text{for all}\: x\in\d M\,,
	$$
	$$
	u(x_j)\,d_{\bar{g}}(x_j,x_k)^{\frac{1}{p-1}}\geq C_0\,,\:\:\:\:\text{for any}\: j\neq k\,,\:j,k=1,...,N\,.
	$$ 
\end{proposition}

\bigskip
We now introduce the notion of an isolated simple blow-up point. If $x_i\to x_0$ is an isolated blow-up point for $\{u_i\}$,  for $0<r<\delta$, set 
$$\bar{u}_i(r)=\frac{2}{\sigma_{n-1}r^{n-1}}\int_{\partial^+B_r^+(0)}(u_i\circ\psi_i)d\sigma_r
\:\:\:\:\text{and}\:\:\:\:w_i(r)=r^{\frac{1}{p_i-1}}\bar{u}_i(r)\,.$$ 
Note that the definition of $w_i$ is invariant under renormalization. More precisely, if $v_i(y)=s^{\frac{1}{p_i-1}}u_i(\psi_i(sy))$, then
$r^{\frac{1}{p_i-1}}\bar{v}_i(r)=(sr)^{\frac{1}{p_i-1}}\bar{u}_i(sr)$.
\begin{definition}\label{def:simples}
	An isolated blow-up point $x_i\to x_0$ for $\{u_i\}$ is {\it{simple}} if there exists $\delta>0$ such that $w_i$ has exactly one critical point in the interval $(0,\delta)$.
\end{definition}
\begin{remark}\label{rk:def:equiv} Let $x_i\to x_0$ be an isolated blow-up point and $R_i\to\infty$. Using Proposition \ref{form:bolha} it is not difficult to see that, choosing a subsequence, $r\mapsto r^{\frac{1}{p_i-1}}\bar{u}_i(r)$ has exactly one critical point in the interval $(0,r_i)$, where $r_i=R_iM_i^{-(p_i-1)}\to 0$. Moreover, its derivative is negative right after the critical point. Hence, if $x_i\to x_0$ is isolated simple then there exists $\delta>0$ such that $w_i'(r)<0$ for all $r\in [r_i,\delta)$. 
\end{remark}

A basic result for isolated simple blow-up point is stated as follows:
\begin{proposition}\label{estim:simples}
	Let $x_i\to x_0$ be an isolated simple blow-up point for $\{u_i\}$.  Then there exist $C,\delta>0$ such that
	\\\\
	(a) $M_iu_i(\psi_i(z))\leq C|z|^{2-n}$\:\:\: for all $z\in B^+_{\delta}(0)\backslash\{0\}$;
	\\\\
	(b) $M_iu_i(\psi_i(z))\geq C^{-1}G_i(z)$\:\:\: for all $z\in B^+_{\delta}(0)\backslash B^+_{r_i}(0)$, where $G_i$ is the Green's function so that:
	
	\begin{align}
	\begin{cases}
	L_{g_i}G_i=0, &\text{in}\; B_{\delta}^+(0)\backslash\{0\},\notag
	\\
	G_i=0, &\text{on}\; \partial^+B_{\delta}^+(0),\notag
	\\
	B_{g_i}G_i=0, &\text{on}\;\partial 'B_{\delta}^+(0)\backslash\{0\}\notag 
	\end{cases}
	\end{align}
	and $|z|^{n-2}G_i(z)\to 1$, as $|z|\to 0$. Here, $r_i$ is defined as in Remark \ref{rk:def:equiv}.
\end{proposition}
\begin{remark}\label{rk:estim:simples}
Suppose that $x_i\to x_0$ is an isolated simple blow-up point for $\{u_i\}$. 
Set $$v_i(y)=M_i^{-1}(u_i\circ\psi_i)(M_i^{-(p_i-1)}y)\,,\:\:\:\:\:\text{for}\: y\in B_{M_i^{p_i-1}\delta}^+(0)\,.$$ 
Then, as a consequence of  Propositions \ref{form:bolha} and \ref{estim:simples}, we see that $v_i\leq CU$ in $B^+_{\delta M_i^{p_1-1}}(0)$.
\end{remark}	

We finally have the following estimate for $\tau_i=\frac{n}{n-2}-p_i$, which is proved using Proposition \ref{Pohozaev}:
\begin{proposition}\label{lim:ei:tau}
	Let $x_i\to x_0$ be an isolated simple blow-up point for $\{u_i\}$ and let $\rho>0$ be small. Then there exists $C>0$ such that
	\ba\label{estim:lim:ei:tau}
	\tau_i\leq
	\begin{cases}
		C\e_i^{1-2\rho+o_i(1)},&\text{for}\:n\geq 4,
		\\
		C\e_i^{1-2\rho+o_i(1)}\log(\e_i^{-1}) ,&\text{for}\:n=3.
	\end{cases}
\end{align}
\end{proposition}

\section{Blow-up estimates}\label{sec:blowup:estim}
In this section we give a pointwise estimate for a blow-up sequence $\{u_i\}$ in a neighborhood of an isolated simple blow-up point. Our estimates are obtained for dimension $n=3$.

Let  $x_i\to x_0$ be an isolated simple blow-up point for the sequence $\{u_i\in\mathcal{M}_i\}$. We use conformal Fermi coordinates centered at $x_i$. Thus we will work with conformal metrics $\tilde{g}_i=\zeta_i^{\frac{4}{n-2}}g_i$ and sequences $\{\tilde{u}_i=\zeta_i^{-1}u_i\}$ and $\{\tilde{\e}_i\}$, where $\tilde{\e}_i=\tilde{u}_i(x_i)^{-(p_i-1)}=\ei$, since $\zeta_i(x_i)=1$. As observed in Remark \ref{rk:mud_conf}, $x_i\to x_0$ is still an isolated blow-up point for the sequence $\{\tilde{u}_i\}$ and satisfies the same estimates of Proposition \ref{estim:simples} (since we have uniform control on the conformal factors $\zeta_i>0$, these estimates are preserved). Let $\fermilinha$ denote the $\tilde{g}_i$-Fermi coordinates centered at $x_i$. 

In order to simplify our notations, we will omit the symbols $\:\tilde{}\:$ and $\psi_i$ in the rest of this section. Thus, the metrics $\tilde{g}_i$ will be denoted by $g_i$ and points $\psi_i(x)\in M$, for $x\in B_{\delta'}^+(0)$, will be denoted simply by $x$.
In particular, $x_i=\psi_i(0)$ will be denoted by $0$ and $u_i\circ\psi_i$ by $u_i$. 

Set $v_i(y)=\ei^{\frac{1}{p_i-1}}u_i(\ei y)$ for $y\in \Beilinha=\Beilinha(0)$. We know that $v_i$ satisfies 
\begin{align}\label{eq:vi'}
\begin{cases}
L_{\hat{g}_i}v_i=0,&\text{in}\:\Beilinha,
\\
B_{\hat{g}_i}v_i+(n-2)\hat{f}_i^{-\tau_i}v_i^{p_i}=0,&\text{on}\:\d '\Beilinha,
\end{cases}
\end{align}
where $\hat{f}_i(y)=f_i(\ei y)$ and $\hat{g}_i$ is the metric with coefficients $(\hat{g}_i)_{kl}(y)=(g_i)_{kl}(\ei y)$.

Let $r\mapsto 0\leq\chi(r)\leq 1$ be a smooth cut-off function such that $\chi(r)\equiv 1$ for $0\leq r\leq \delta$ and $\chi(r)\equiv 0$ for $r>2\delta$. 
We set $\chi_{\e}(r)=\chi(\e r)$.
Thus,  $\chi_{\e}(r)\equiv 1$ for $0\leq r\leq \delta\e^{-1}$ and $\chi_{\e}(r)\equiv 0$ for $r>2\delta\e^{-1}$.

Observing that $tr(\pi_{kl}(0))=0$ holds due to Proposition \ref{conf:fermi:thm},  by \cite[Proposition 5.1]{almaraz3} for every $i$ there is a solution $\phi_i$ of 
\begin{align}
\begin{cases}\label{linear:8}
\Delta\phi_{i}(y)=-2\chi_{\ei}(|y|)\ei \pi_{kl}(0)y_n(\d_k\d_l U)(y)\,,&\text{for}\:y\in\Rn\,,
\\
\d_n\phi_{i}(\bar{y})+nU^{\frac{2}{n-2}}\phi_{i}(\bar{y})=0\,,&\text{for}\:\bar{y}\in\d\Rn\,,
\end{cases}
\end{align} 
where $\Delta$ stands for the Euclidean Laplacian, satisfying 
\begin{equation}\label{estim:phi'}
|\nabla^r\phi_i|(y)\leq C\ei |\pi_{kl}(0)|(1+|y|)^{3-r-n}\,,\:\:\:\:\text{for}\: y\in\Rn\,,\,r=0,1 \:\text{or}\:2\:, 
\end{equation} 
\begin{equation}\label{hip:phi}
\phi_i(0)=\frac{\d\phi_i}{\d y_1}(0)=...=\frac{\d\phi_{i}}{\d y_{n-1}}(0)=0
\end{equation}
and
\begin{equation}\label{orto-phi}
\int_{\d\Rn}U^{\frac{n}{n-2}}(\bar{y})\phi_i(\bar{y})\,d\bar{y}=0\,.
\end{equation}

\bigskip\noindent
{\bf{Assumption}} In the rest of this section, $n=3$.
\begin{lemma}\label{estim:blowup1}
There exist $\delta, C>0$ such that, for $|y|\leq\delta\ei^{-1}$,
$$|v_i-U-\phi_i|(y)\leq C\max\{\ei,\tau_i\}\,.$$
\end{lemma}
\begin{proof}
We consider $\delta<\delta'$ to be chosen later and set 
$$\Lambda_i=\max_{|y|\leq \delta\ei^{-1}} |v_i-U-\phi_i|(y)=|v_i-U-\phi_i|(y_i)\,,$$ 
for some $|y_i|\leq \delta\ei^{-1}$. From Remark \ref{rk:estim:simples}
we know that $v_i(y)\leq CU(y)$ for $|y|\leq \delta\ei^{-1}$. Hence, if there exists $c>0$ such that $|y_i|\geq c\ei^{-1}$, then
$$\Lambda_i=|v_i-U-\phi_i|(y_i)\leq C\,|y_i|^{2-n}\leq C\,\ei^{n-2}.$$
This implies the stronger inequality $|v_i-U-\phi_i|(y)\leq C\,\epsilon_i^{n-2}=C\epsilon_i$, for $|y|\leq \delta\ei^{-1}$. Hence, we can suppose that $|y_i|\leq \delta\ei^{-1}/2$. 

Suppose, by contradiction, the result is false. 
Then, choosing a subsequence if necessary, we can suppose that
\begin{equation}
\label{hipLambda}
\Lambda_i^{-1}\ei\to 0\:\:\:\:\text{and}\:\:\:\:\Lambda_i^{-1}\tau_i\to 0\,.
\end{equation}
Define
$$w_i(y)=\Lambda_i^{-1}(v_i-U-\phi_i)(y)\,,\:\:\:\:\text{for}\:\: |y|\leq \delta\ei^{-1}\,.$$
By the equations \eqref{eq:U} and \eqref{eq:vi'}, $w_i$ satisfies
\begin{equation}\label{wi}
\begin{cases}
L_{\hat{g}_i}w_i=Q_i\,,&\text{in}\:\Bei\,,\\
B_{\hat{g}_i}w_i+b_iw_i=\overline{Q}_i\,,&\text{on}\:\partial '\Bei\,,
\end{cases}
\end{equation}
where 

\begin{flushleft}
$b_i=(n-2)\hat{f}_i^{-\tau_i}\frac{v_i^{p_i}-(U+\phi_i)^{p_i}}{v_i-(U+\phi_i)}$,\\
$Q_i=-\Lambda_i^{-1}\big\{(L_{\hat{g}_i}-\Delta)(U+\phi_i)+\Delta\phi_i\big\}$,\\
$\overline{Q}_i=-\Lambda_i^{-1}\left\{(n-2)\hat{f}_i^{-\tau_i}(U+\phi_i)^{p_i}
-(n-2)U^{\frac{n}{n-2}}-nU^{\frac{2}{n-2}}\phi_i-\frac{n-2}{2}h_{\hat{g}_i}(U+\phi_i)\right\}$.
\end{flushleft}

Observe that, for any function $u$,
\begin{align}
(L_{\hat{g}_i}-\Delta)u(y)&=(\hat{g}^{kl}_i-\delta^{kl})(y)\d_k\d_lu(y)
+(\d_k\hat{g}^{kl}_i)(y)\d_lu(y)\notag
\\
&\hspace{2cm}-\frac{n-2}{4(n-1)}R_{\hat{g}_i}(y)u(y)
+\frac{\d_k \sqrt{\det \hat{g}_i}}{\sqrt{\det \hat{g}_i}}\hat{g}_i^{kl}(y)\d_lu(y)\notag
\\
&=(g^{kl}_i-\delta^{kl})(\epsilon_i y)\d_k\d_lu(y)
+\epsilon_i (\d_kg^{kl}_i)(\epsilon_i y)\d_lu(y)\notag
\\
&\hspace{2cm}-\frac{n-2}{4(n-1)}\epsilon_i^2R_{g_i}(\epsilon_i y)u(y)
+O(\ei^N|y|^{N-1})\d_lu(y)\,.\notag
\end{align}
Hence,
\begin{align}\label{Qi}
Q_i(y)
&=-\Lambda_i^{-1}(g^{kl}_i-\delta^{kl})(\epsilon_i y)\d_{k}\d_l (U+\phi_i)(y)
-\Lambda_i^{-1}\epsilon_i (\d_kg^{kl}_i)(\epsilon_i y)\d_l (U+\phi_i)(y)\notag
\\
&\hspace{0.5cm}+\frac{n-2}{4(n-1)}\Lambda_i^{-1}\epsilon_i^2R_{g_i}(\epsilon_i y)(U+\phi_i)(y)
-\Lambda_i^{-1}\Delta\phi_i(y)\notag
\\
&\hspace{0.5cm}+O\left(\Lambda_i^{-1}\ei^N|y|^{N-1}(1+|y|)^{1-n}\right)\notag
\\
&=O\left(\Lambda_i^{-1}\ei^N(1+|y|)^{N-n}\right)+O\left(\Lambda_i^{-1}\ei^2(1+|y|)^{2-n}\right)\,.
\end{align}

Observe that
\ba
&(n-2)\hat{f}_i^{-\tau_i}(U+\phi_i)^{p_i}
-(n-2)U^{\frac{n}{n-2}}-nU^{\frac{2}{n-2}}\phi_i\notag
\\
&\hspace{1cm}=(n-2)\left(\hat{f}_i^{-\tau_i}(U+\phi_i)^{p_i}-(U+\phi_i)^{\frac{n}{n-2}}\right)
+O(U^{\frac{4-n}{n-2}}\phi_i^2)\notag
\\
&\hspace{1cm}=(n-2)\hat{f}_i^{-\tau_i}\left((U+\phi_i)^{p_i}-(U+\phi_i)^{\frac{n}{n-2}}\right)\notag
\\
&\hspace{2cm}+(n-2)(\hat{f}_i^{-\tau_i}-1)(U+\phi_i)^{\frac{n}{n-2}}+O(U^{\frac{4-n}{n-2}}\phi_i^2)\,.\notag
\end{align}
Using

\begin{flushleft}
$U^{\frac{4-n}{n-2}}\phi_i^2=O(\ei^2|\pi_{kl}(0)|^2(1+|y|)^{2-n})$,\\
$h_{\hat{g}_i}(U+\phi_i)=O(\ei^N(1+|y|)^{N+1-n})$,\\
$\hat{f}_i^{-\tau_i}\left((U+\phi_i)^{p_i}-(U+\phi_i)^{\frac{n}{n-2}}\right)
=O(\tau_i(U+\phi_i)^{\frac{n}{n-2}}\log(U+\phi_i))=O(\tau_i(1+|y|)^{1-n})$,\\
$(\hat{f}_i^{-\tau_i}-1)(U+\phi_i)^{\frac{n}{n-2}}=O(\tau_i\log(f_i)(U+\phi_i)^{\frac{n}{n-2}})
=O(\tau_i(1+|y|)^{-n})$,
\end{flushleft}
where in the second line we used Proposition \ref{conf:fermi:thm},
we obtain
\begin{equation}\label{barQi}
\bar{Q}_i(\bar{y})= O\left(\Lambda_i^{-1}\ei^2(1+|\bar{y}|)^{2-n}\right)+
O\left(\Lambda_i^{-1}\tau_i(1+|\bar{y}|)^{1-n}\right)\,.
\end{equation}

Moreover,
\begin{equation}
\label{lim:bi}
b_i(y)\to nU^{\frac{2}{n-2}}\,,\:\:\:\:\text{in}\: C^2_{loc}(\Rn)\,,
\end{equation}
and
\begin{equation}
\label{estim:bi}
b_i(y)\leq C(1+|y|)^{-2}\,,\:\:\:\:\text{for}\: |y|\leq\delta\ei^{-1}\,. 
\end{equation}

Since $|w_i|\leq |w_i(y_i)|=1$, we can use standard elliptic estimates to conclude that $w_i\to w$, in $C_{loc}^2(\mathbb{R}_+^n)$, for some function $w$, choosing a subsequence if necessary. From the identities \eqref{hipLambda}, \eqref{Qi}, \eqref{barQi} and  \eqref{lim:bi}, we see that $w$ satisfies
\begin{equation}\label{w}
\begin{cases}
\Delta w=0\,,&\text{in}\:\mathbb{R}_+^n\,,\\
\d_n w+nU^{\frac{2}{n-2}}w=0\,,&\text{on}\:\partial\mathbb{R}_+^n\,.
\end{cases}
\end{equation}

\bigskip
\noindent
{\it{Claim.}} $\:\:w(y)=O((1+|y|)^{-1})$, for $y\in \Rn$.

\vspace{0.2cm}
Choosing $\delta>0$ sufficiently small, we can consider the Green's function $G_i$ for the conformal Laplacian $L_{\hat{g}_i}$ in $\Bei$ subject to the boundary conditions $B_{\hat{g}_i} G_i=0$ on $\partial'\Bei$ and $G_i=0$ on $\partial^+\Bei$. Let $\eta_i$ be the inward unit normal vector to $\partial^+\Bei$. Then the Green's formula gives
\begin{align}\label{wG}
w_i(y)&=-\int_{\Bei}G_i(\xi,y)Q_i(\xi) \,dv_{\hat{g}_i}(\xi)
+\int_{\partial^+\Bei}\frac{\partial G_i}{\partial\eta_i}(\xi,y)w_i(\xi)\,d\sigma_{\hat{g}_i}(\xi)\notag
\\
&\hspace{1cm}
+\int_{\partial'\Bei}G_i(\xi,y)\left(b_i(\xi)w_i(\xi)-\overline{Q}_i(\xi)\right)\,d\sigma_{\hat{g}_i}(\xi)\,.
\end{align}
Using the estimates \eqref{Qi}, \eqref{barQi} and \eqref{estim:bi}  in the equation \eqref{wG}, we obtain
\begin{align}
|w_i(y)|
\leq &\:C\Lambda_i^{-1}\ei^2\int_{\Bei}|\xi-y|^{2-n}(1+|\xi|)^{2-n}d\xi\notag
\\
&+C\int_{\d'\Bei}|\bar{\xi}-y|^{2-n}(1+|\bar{\xi}|)^{-2}d\bar{\xi}
+C\Lambda_i^{-1}\ei^2
\int_{\d'\Bei}|\bar{\xi}-y|^{2-n}(1+|\bar{\xi}|)^{2-n}d\bar{\xi}\notag
\\
&+C\Lambda_i^{-1}\tau_i\int_{\d'\Bei}|\bar{\xi}-y|^{2-n}(1+|\bar{\xi}|)^{1-n}d\bar{\xi}
+C\Lambda_i^{-1}\ei^{n-2}\int_{\d^+\Bei}|\xi-y|^{1-n}d\sigma(\xi)\,,\notag
\end{align}
for $|y|\leq \delta\ei^{-1}/2$. Here, we have used the fact that $|G_i(x,y)|\leq C\,|x-y|^{2-n}$ for $|y|\leq \delta\ei^{-1}/2$ and, since $v_i(y)\leq CU(y)$, $|w_i(y)|\leq C\Lambda_i^{-1}\ei^{n-2}$ for $|y|=\delta\ei^{-1}$. Hence, 
$$
|w(y)|\leq C\Lambda_i^{-1}\ei^2(\delta\ei^{-1})^{4-n}+C(1+|y|)^{-1}+C\Lambda_i^{-1}\ei^2\log(\delta\ei^{-1})
+C\Lambda_i^{-1}\tau_i(1+|y|)^{2-n}+C\Lambda_i^{-1}\ei^{n-2}.
$$
Since $n=3$, this gives
\begin{equation}\label{estim:wi}
|w_i(y)|\leq C\,\left((1+|y|)^{-1}+\Lambda_i^{-1}\ei+\Lambda_i^{-1}\tau_i\right)
\end{equation}
for $|y|\leq \delta\ei^{-1}/2$.
The Claim now follows from the hypothesis (\ref{hipLambda}).

Now, we can use the claim above and Lemma \ref{classifLinear} to see that 
$$w(y)=\sum_{j=1}^{n-1}c_j\d_jU(y)
+c_n\left(\frac{n-2}{2}U(y)+y^b\partial_b U(y)\right)\,,$$
for some constants $c_1,...,c_n$.
It follows from the identity (\ref{hip:phi}) that $w_i(0)=\frac{\partial w_i}{\partial y_j}(0)=0$ for $j=1,...,n-1$. Thus we conclude that $c_1=...=c_n=0$. Hence, $w\equiv 0$. Since $|w_i(y_i)|=1$, we have $|y_i|\to\infty$. This, together with the hypothesis (\ref{hipLambda}), contradicts the estimate (\ref{estim:wi}), since $|y_i|\leq \delta\ei^{-1}/2$, and concludes the proof of Lemma \ref{estim:blowup1}.
\ep
\begin{lemma}\label{estim:tau}
There exists $C>0$ such that $\tau_i\leq C\ei$.
\end{lemma}
\begin{proof}
Suppose, by contradiction, the result is false. Then we can suppose that $\tau_i^{-1}\ei\to 0$
and, by Lemma \ref{estim:blowup1}, there exists $C>0$ such that
$$|v_i-U-\phi_i|(y)\leq C\tau_i\,,\:\:\:\:\text{for}\:\:|y|\leq\delta\ei^{-1}\,.$$
Define
$$w_i(y)=\tau_i^{-1}(v_i-U-\phi_i)(y)\,,\:\:\:\:\text{for}\:\: |y|\leq \delta\ei^{-1}\,. $$
Then $w_i$ satisfies the equations \eqref{wi} with 
\begin{flushleft}
$b_i=(n-2)\hat{f}_i^{-\tau_i}\frac{v_i^{p_i}-(U+\phi_i)^{p_i}}{v_i-(U+\phi_i)}$,\\
$Q_i=-\tau_i^{-1}\left\{(L_{\hat{g}_i}-\Delta)(U+\phi_i)+\Delta \phi_i\right\}$,\\
$\overline{Q}_i=-\tau_i^{-1}\left\{(n-2)\hat{f}_i^{-\tau_i}(U+\phi_i)^{p_i}
-(n-2)U^{\frac{n}{n-2}}-nU^{\frac{2}{n-2}}\phi_i-\frac{n-2}{2}h_{\hat{g}_i}(U+\phi_i)\right\}$.
\end{flushleft}
Similarly to the estimates \eqref{Qi} and \eqref{barQi}  we have
\begin{align}
|Q_i(y)|&\leq C\tau_i^{-1}\ei^2(1+|y|)^{2-n}\,,\label{Qi2}
\\
|\overline{Q}_i(y)|&\leq C\tau_i^{-1}\ei^2(1+|y|)^{2-n}+
C(1+|y|)^{1-n}\label{barQi2}
\end{align}
and $b_i$ satisfies the estimate \eqref{estim:bi}.

By definition, $w_i\leq C$ and, by elliptic standard estimates, we can suppose that $w_i\to w$ in $C^2_{loc}(\mathbb{R}^n_+)$ for some function $w$. By the identity (\ref{lim:bi}) and the estimates \eqref{Qi2} and \eqref{barQi2} we see that $w$ satisfies the equations \eqref{w}.

A contradiction is achieved following the same lines as \cite[Lemma 6.2]{almaraz3}.

\end{proof}
\begin{proposition}\label{estim:blowup:compl}
There exist $C,\delta>0$ such that 
$$
|\nabla^k(v_i-U-\phi_i)(y)|\leq C\ei(1+|y|)^{-k}
$$ 
for all $|y|\leq\delta\ei^{-1}$ and $k=0,1,2$.
\end{proposition}
\bp
The estimate with $k=0$ follows from Lemmas \ref{estim:blowup1} and \ref{estim:tau}.
The estimates with $k=1,2$ follow from elliptic theory.
\ep

\section{The Pohozaev sign restriction}\label{sec:sign:restr}
In this section we assume $n=3$ and prove a sign restriction for an integral term in Proposition \ref{Pohozaev} and some consequences for the blow-up set. 

\begin{theorem}\label{cond:sinal}
Let $x_i\to x_0$ be an isolated simple blow-up point for the sequence $\{u_i\in\mathcal{M}_i\}$. 
Suppose that $u_i(x_i)u_i\to G$ away from $x_0$, for some function $G$. Then
\begin{equation}\label{eq:cond:sinal}
\liminf_{r\to 0}P'(G,r)\geq 0\,.
\end{equation}
\end{theorem}

\bp


Set $v_i(y)=\ei^{\frac{1}{p_i-1}}u_i(\ei y)$ for $y\in \Bei=\Bei(0)$. We know that $v_i$ satisfies 
\begin{align}\notag
\begin{cases}
L_{\hat{g}_i}v_i=0,&\text{in}\:\Bei,
\\
B_{\hat{g}_i}v_i+(n-2)\hat{f}_i^{-\tau_i}v_i^{p_i}=0,&\text{on}\:\d '\Bei,
\end{cases}
\end{align}
where $\hat{f}_i(y)=f_i(\ei y)$ and $\hat{g}_i$ is the metric with coefficients $(\hat{g}_i)_{kl}(y)=(g_i)_{kl}(\ei y)$.
Observe that, from Remark \ref{rk:estim:simples}, we know that $v_i\leq CU$ in $B^+_{\delta\ei^{-1}}$.

We write the Pohozaev identity of Proposition \ref{Pohozaev} as
\begin{equation}\label{Pohoz}
P(u_i,r)=F_i(u_i,r)+\bar{F}_i(u_i,r)+\frac{\tau_i}{p_i+1}Q_i(u_i,r)\,,
\end{equation}
where
\\
$F_i(u,r)=-\int_{B_r^+}(z^b\partial_bu+\frac{n-2}{2}u)(L_{g_i}-\Delta)u\,dz$,
\\
$\bar{F}_i(u,r)=\frac{n-2}{2}\int_{\partial 'B_r^+}(\bar{z}^b\partial_bu+\frac{n-2}{2}u)h_{g_i}u\,d\bar{z}$,
\\
$Q_i(u,r)=
\frac{(n-2)^2}{2}\int_{\d 'B_r^+}f_i^{-\tau_i}u^{p_i+1}d\bar{z}
-(n-2)\int_{\d 'B_r^+}(\bar{z}^k\partial_kf)f_i^{-\tau_i-1}u^{p_i+1}d\bar{z}$.

Since we can assume $h(0)=0$, we have 
$$\bar F_i(u_i,r)=O(\e^{n-2}r)\,.$$ 
On the other hand, we can choose $r>0$ small such that $Q_i(u_i,r)\geq 0$. 
So we only have to handle $F_i(u_i,r)$.

Set $\check{U}_i(z)=\ei^{-\frac{1}{p_i-1}}U(\ei^{-1}z)$ and $\check{\phi}_i(z)=\ei^{-\frac{1}{p_i-1}}\phi_i(\ei^{-1}z)$. We have
\begin{align}
F_i(u_i,r)=&-\int_{B_r^+}(z^b\partial_bu_i
+\frac{n-2}{2}u_i)(L_{g_i}-\Delta)u_idz\notag
\\
&=-\epsilon_i^{-\frac{2}{(p_i-1)}+n-2}\int_{B_{r\epsilon_i^{-1}}^+}(y^b\partial_bv_i+\frac{n-2}{2}v_i)\notag
(L_{\hat{g}_i}-\Delta)v_idy\,,
\notag
\end{align}
\begin{align}
F_i(\check{U}_i+\check{\phi}_i,r)
&=-\int_{B_r^+}(z^b\partial_b\check{U}_i
+\frac{n-2}{2}\check{U}_i)(L_{g_i}-\Delta)\check{U}_idz\notag
\\
&=-\epsilon_i^{-\frac{2}{(p_i-1)}+n-2}\int_{B_{r\epsilon_i^{-1}}^+}\left(y^b\partial_b(U+\phi_i)
+\frac{n-2}{2}(U+\phi_i)\right)\notag
(L_{\hat{g}_i}-\Delta)(U+\phi_i)dy\,.
\notag
\end{align}
Observe that $\ei^{-\frac{2}{p_i-1}+n-2}=\ei^{-(n-2)\frac{\tau_i}{p_i-1}}\to 1$, as $i\to\infty$, by Proposition \ref{lim:ei:tau}.

Now we use that  $n=3$.
It follows from Proposition \ref{estim:blowup:compl} that 
\begin{equation}\label{approx:F}
|F_i(u_i,r)-F_i(\check{U}_i+\check{\phi}_i,r)|\leq C \ei^2\int_{B_{r\epsilon_i^{-1}}^+}(1+|y|)^{1-n}dy
\leq C\ei r\,.
\end{equation}
We know from \eqref{exp:g} that
$g^{kl}(z)=\delta_{kl}+2\pi_{kl}(0)z_n+O(|z|^2)$
in Fermi coordinates,
and recall that we are assuming $\text{tr}(\pi_{kl}(0))=h(0)=0$. Thus, due to symmetry arguments,
$$
F_i(\check{U}_i+\check{\phi}_i,r)=O(\ei r)\,.
$$
Hence, $P(u_i,r)\geq -C\ei r$, which implies that
$$
P'(G,r)=\lim_{i\to \infty}\ei^{-\frac{2}{p_i-1}}P(u_i,r)\geq -Cr\,.
$$
\ep

Once we have proved Theorem \ref{cond:sinal}, the next two propositions are similar to \cite[Lemma 8.2, Proposition 8.3]{khuri-marques-schoen} or \cite[Propositions 4.1 and 5.2]{li-zhu2}.
\begin{proposition}\label{isolado:impl:simples}
	 Let $x_i\to x_0$ be an isolated  blow-up point for the sequence $\{u_i\in\mathcal{M}_i\}$. Then $x_0$ is an isolated simple blow-up point for $\{u_i\}$.
\end{proposition}
\begin{proposition} \label{dist:unif}
	Let $\b, R, u, C_0(\b,R)$ and $\{x_1,...,x_N\}\subset \d M$ be as in Proposition \ref{conj:isolados}. If $\b$ is sufficiently small and $R$ is  sufficiently large, then there exists a constant $\bar C(\b,R)>0$ such that if  $\max_{\d M}u\geq C_0$ then 
	 $$d_{\bar g}(x_j,x_k)\geq \bar C\:\:\:\:\:\text{for all}\:1\leq j\neq k\leq N.$$
\end{proposition}
\begin{corollary}\label{Corol:8.4}
Suppose the sequence  $\{u_i\in\mathcal{M}_i\}$ satisfies $\max_{\d M} u_i\to\infty$. Then $p_i\to n/(n-2)$ and the set of blow-up points is finite and consists only of isolated simple blow-up points.
\end{corollary}


\section{Proof of Theorem \ref{compactness:thm}}\label{sec:pf:thm}
In view of standard elliptic estimates and Harnack inequalities, we only need to prove that $\|u\|_{C^0(\d M)}$ is bounded from above (see \cite[Lemma A.1]{han-li} for the boundary Harnack inequality). Assume by contradiction there exists a sequence $u_i$ of positive solutions of (\ref{main:eq}) such that 
    \begin{equation*}
    \max_{\d M} u_i\to\infty\:\:\:\:\:\text{as}\:i\to\infty.
    \end{equation*}
    It follows from Corollary \ref{Corol:8.4} that we can assume $u_i$ has $N$ isolated simple blow-up points 
    \begin{align*}
    x_i^{(1)}\to x^{(1)},\:...\:,\: x_i^{(N)}\to x^{(N)},
    \end{align*}
     and that $\tau_i=\frac{n}{n-2}-p_i\to 0$ as $i\to\infty$. Without loss of generality, suppose 
     \begin{align*}
     u_i(x_i^{(1)})=\min\big\{u_i(x_i^{(1)}), ..., u_i(x_i^{(N)})\big\}\:\:\:\:\text{for all}\:i.
     \end{align*}
    
    Now for each $k=1,...,N$, consider the Green's function $G_k$ for the conformal Laplacian $L_{g}$ with boundary condition $B_{g}G_k=0$ and singularity at $x^{(k)}\in \d M$. In Fermi coordinates centered at the respective singularities, those functions satisfy
    \begin{align*}
    \big|G_k(x)-|x|^{2-n}\big|\leq C\big|\log|x|\,\big|\:\:\:\:\text{for}\:n=3,
    \end{align*}
according to \cite[Proposition B.2]{almaraz-sun}.

It follows from the upper bound (a) of Proposition \ref{estim:simples} that there exists some function $G$ such that $u_i(x_i^{(1)})u_i\to G$ in $C^2_{\text{loc}}(M\backslash  \{x^{(1)},...,x^{(N)}\})$. Moreover, the lower control (b) of that proposition and elliptic theory yields the existence of $a_k>0$, $k=1,...,N$, and $b\in C^2(M)$ such that 
    \begin{equation*}
    G=\sum_{k=1}^{N}a_kG_k+b,
    \end{equation*}
 and 
    \begin{align*}
    \begin{cases}
    L_{g}b=0,&\text{in}\:M,
    \\
    B_{g}b=0,&\text{on}\:\d M.
    \end{cases}
    \end{align*}

    The hypothesis $Q(M,\d M)>0$ ensures that $b\equiv 0$. If $\hat g=G_1^{\frac{4}{n-2}}g$, by Proposition \ref{propo:I:mass}, $(M\backslash\{x^{(1)}\}, \hat g)$ is an asymptotically flat manifold (in the sense of Definition \ref{def:asym}) with mass 
    $$
    m(\hat g)=\lim_{\rho\to 0} I(x^{(1)}, \rho).
    $$
Moreover, we have
$R_{\hat g}=-\frac{4(n-1)}{n-2}G^{\frac{n+2}{n-2}}L_gG=0$ and $h_{\hat g}=-\frac{2}{n-2}G^{\frac{n}{n-2}}B_gG=0$.
Then the positive mass Theorem \ref{pmt} and the assumption that $M$ is not conformally equivalent to $B^3$ gives $m(\hat g)>0$. So, by  Proposition \ref{I:P'},
    $$
    \lim_{\rho\to 0}P'(G_1, \rho)<0.
    $$
     This contradicts the local sign restriction of Theorem \ref{cond:sinal} and ends the proof of Theorem \ref{compactness:thm}.

\bigskip\noindent
{\bf{Acknowledgment.}} The authors would like to thank the anonymous referee for his/her valuable comments and suggestions.

\bigskip\noindent
\textsc{S\'ergio Almaraz\\
Instituto de Matem\'atica e Estat\' istica, 
Universidade Federal Fluminense\\
Rua Prof. Marcos Waldemar de Freitas S/N,
Niter\'oi, RJ,  24.210-201, Brazil}\\
e-mail: {\bf{sergio.m.almaraz@gmail.com}}

\bigskip\noindent
\textsc{Olivaine S. de Queiroz\\
Departamento de Matem\'atica,
Universidade Estadual de Campinas - IMECC\\
Rua S\'ergio Buarque de Holanda, 651, 
Campinas, SP, 13083-859, Brazil}\\
e-mail: {\bf{olivaine@ime.unicamp.br}}

\bigskip\noindent
\textsc{Shaodong Wang\\ 
Department of Mathematics and Statistics,
McGill University\\ 
805 Sherbrooke Street West,
Montreal, Quebec H3A 0B9, Canada\\}
e-mail: {\bf{shaodong.wang@mail.mcgill.ca}}

\end{document}